\documentclass[11pt]{amsart}
\usepackage{amssymb}
\numberwithin{equation}{section}

\newtheorem{Theorem}{Theorem}[section]
\newtheorem{Lemma}[Theorem]{Lemma}
\newtheorem{Claim}[Theorem]{Claim}
\newtheorem{Definition}[Theorem]{Definition}
\newtheorem{Corollary}[Theorem]{Corollary}
\newtheorem{Question}[Theorem]{Question}

\newtheorem{Remark}[Theorem]{Remark}
\newtheorem{Fact}[Theorem]{Fact}
 
\def\qed{\hfill $\square$}
\def\demo{\indent{\bf Proof. }}
\def\enddemo{\qed}

\def\cl{cellular-Lindel\" of}

 \def\be{\begin{enumerate}}
 \def\ee{\end{enumerate}}

\begin{document}

\title[Productivity of Cellular-Lindel\" of spaces]{\Large Productivity
 of cellular-Lindel\" of spaces}
  
\author[Alan Dow and R.M. Stephenson, Jr.]{\bf Alan Dow
 and R.M. Stephenson, Jr.}

\address{
Department of Mathematics and Statistics,
University of North Carolina at Charlotte,
Charlotte, NC 28223}
\email{adow@uncc.edu}
 
 \address{
Department of Mathematics,
University of South Carolina,
Columbia, SC 29208}
\email{stephenson@math.sc.edu} 


  
 \begin{abstract}  The main purpose of this note is to prove that the
product of a \cl\ space with a space of countable spread
need not be cellular-Lindel\" of.
 \end{abstract} 
\maketitle
 
 \section{Introduction}\label{sect1} 

All hypothesized spaces are Hausdorff, and by
 {\it Lindel\" of} ({\it of countable spread})
we shall mean any space in which every open cover has a countable
subcover (every discrete subspace is countable);
 otherwise, the terminology and notation not defined here
generally agree with that in  \cite{E} or \cite{K}.

 A. Bella and S. Spadaro in \cite{BS1} 
defined a space $X$ to be {\it cellular-Lindel\" of} provided that for every
family $\mathcal U$ of pairwise disjoint nonempty open subsets of $X$ there
 is a
 Lindel\" of subspace $L\subset X$ such that $L\cap U\neq\emptyset$ for every
$U\in{\mathcal U}$. In that article and also in \cite{BS2} they derived
some properties of \cl\ spaces. Likewise, W.\ Xuan and
 Y.\ Song studied \cl\ spaces in \cite{XS} and several
 other articles referred to in \cite{XS}. A similar family of spaces,
the cellular-compact spaces, were introduced and studied extensively by
V.V. Tkachuk and R.G. Wilson in \cite{TW}. A space $X$ is said to be
 {\it cellular-compact} provided that for every family $\mathcal U$ of
pairwise disjoint nonempty open subsets of $X$ there is a
compact subspace $K\subset X$ such that $K\cap U\neq\emptyset$ for every
$U\in{\mathcal U}$.

 Among the nice properties of these two cellular concepts, most of which were
noted or shown
in several of the articles above, are: each is preserved by regular closed
subsets, extensions, continuous images, and finite unions (and for
the case \cl, countable unions and open $F_\sigma$-subsets).

In \cite{XS} and especially in \cite{TW} its authors proved that
 dense subspaces of compact product
 spaces, such as $\Sigma$-product subspaces, 
provide a useful source of examples of non-Lindel\" of spaces illustrating
properties of cellular-compact or \cl\ spaces.  In \cite{XS}
its authors presented several theorems, examples and questions 
concerning products of \cl\ spaces.  We list some of these
results next, but first note one obvious necessary condition for a
 product space to be cellular-compact or \cl,
according to the invariance under continuous maps property.

\begin{Theorem} If a product space is cellular-compact or \cl,
then each of its factor spaces must have the same property.
\end{Theorem}

In \cite{RT} A.D. Rojas-S\' anchez and A. Tamariz-Mascar\' ua provided a
very clever proof that there exist two Lindel\" of spaces whose product
 contains an uncountable clopen discrete subspace $D$. The authors of
 \cite{XS} referred the reader to \cite{RT}, re-described that example,
 re-derived the properties of $D$, and noted that it establishes the following:

\begin{Theorem} The product of two Tychonoff Lindel\" of spaces need not
be \cl.
\end{Theorem}

They then obtained the following result.

\begin{Theorem}\label{thm1.2}{\rm(\cite{XS}, Theorem 3.9)}
 The product 
of a separable space  and a cellular-Lindel\" of space is
\cl.
\end{Theorem}
 
\demo
We sketch a different proof from the one in \cite{XS}:  It is immediate
that the product of a countable space and a \cl\ space
is \cl, and hence any extension of such a space
must also be \cl.   
\enddemo
\smallskip
  
Noting that this theorem implies 
the product of a compact metrizable space with a 
\cl\ space is \cl, 
 the authors of \cite{XS} asked:

\begin{Question} Is the product of a compact space and a \cl\ 
space \cl?
\end{Question}

In Section 2 we shall provide a negative answer to this question, and we shall
answer an analogous question not raised in \cite{TW} or \cite{XS} about
whether or not the property cellular-compact is productive.  It will be
 shown that the product
 of a convergent sequence with a cellular-compact space need not be
 cellular-compact.

The authors of \cite{XS}  stated in their Theorem 3.12
 that the product of a \cl\  space with a space of countable
 spread, such as a
hereditarily Lindel\" of space, is \cl.  However, 
 as noted
in the Abstract, the main purpose of this note is to provide a proof that
 this assertion is not true. In Section 3
the focus will be on properties possible counterexamples will need to possess.
A proof will be given in Section 4
 showing it is consistent that the product of a \cl\ space
 with a hereditarily Lindel\" of space need not be \cl. Then in Section 5
 a construction of Justin Moore will be  used to produce very
different counterexamples within ZFC.
 
 \section{Products with one factor compact}\label{sect2} 

 Let us first recall some needed terminology and a theorem.
A family $\mathcal V$ of nonempty open subsets of a space
$X$ is called a {\it $\pi$-base for $X$} (a {\it local
$\pi$-base for a point $x\in X$}) provided that for every nonempty
open subset $T$ of $X$ (containing $x$) there exists $V\in{\mathcal V}$
such that $V\subset T$. The $\pi$-{\it weight} $\pi(X)$ of the space $X$
  ($\pi$-character $\pi\chi(x,X)$ of the
space $X$ at the point $x$) is the minimal cardinality of a
$\pi$-base for $X$ (a local $\pi$-base at $x$).
 The {\it density} ${\rm d}(X)$ of $X$ is the minimal cardinality of a
dense subset of $X$.
 A point $p\in\beta X\setminus X$ is called a
{\it remote point of a Tychonoff space $X$} if
 $p\notin{\rm cl}_{\beta X}(Y)$ for any nowhere dense set $Y\subset X$.
Finally,
 (Theorem 3.13 of \cite{TW}) if $X$ is a regular cellular-compact
 space and $p$ is a non-isolated point of $X$, then the space
 $X\setminus\{p\}$ is cellular-compact if and only if there exists no
pairwise disjoint local $\pi$-base at the point $p$ in $X$.  We shall
use only a special case of a corollary to this theorem (Corollary 3.14 of \cite{TW}), a case which by itself has an immediate proof,
and so we give both here.

\begin{Corollary}{\rm(\cite{TW})} If $K$ is a compact space,
 $p$ is a non-isolated
point of $K$, and there exists no pairwise disjoint local $\pi$-base at
 the point $p$ in $K$, then the space $X=K\setminus\{p\}$ is cellular-compact.
\end{Corollary}

\demo
Assume $\mathcal U$ is a family of pairwise disjoint nonempty open subsets
of $K$.  Choose an open neighborhood $W$ of $p$ which does not contain any
member of $\mathcal U$. Then $K\setminus W=X\setminus W$ is a compact set
which meets every member of $\mathcal U$.
\enddemo

\begin{Theorem} Let $X=\beta{\mathbb Q}\setminus\{p\}$, where $p$ is any
remote point of $\mathbb Q$, and
let $Y={\mathbb N}\cup\{\infty\}$ be the one-point compactification of
 $\mathbb N$. Then the space $X$ is cellular-compact, but the product space
 $X\times Y$ is not cellular-compact.
\end{Theorem}

\demo
To prove that $X$ is cellular-compact, it suffices by Corollary 2.1 for us to prove
that there is no pairwise disjoint local $\pi$-base for $p$
 in $\beta{\mathbb Q}$.
Let $\mathcal V$ be an arbitrary family of pairwise disjoint
 nonempty open sets in $\beta{\mathbb Q}$.  Proceeding as in Example 3.21 of \cite{TW}, we may
(and do) choose, for each $V\in\mathcal V$,  a rational number $q_V\in V$.
 Then $D=\{q_V:V\in{\mathcal V}\}$ is discrete, and so the remote point
 $p\notin{\rm cl}_{\beta {\mathbb Q}}(D)$. Hence
 $\beta{\mathbb Q}\setminus{\rm cl}_{\beta {\mathbb Q}}(D)$ is an
 open neighborhood of
$p$ which contains no member of $\mathcal V$. Therefore, $\mathcal V$ cannot be a local $\pi$-base for
$p$ in $\beta{\mathbb Q}$.

Next, again using the denseness of $\mathbb Q$ in $\beta\mathbb Q$,
 let ${\mathcal W}=\{W_n:n\in{\mathbb N}\}$ be a $\pi$-base for
 $\beta{\mathbb Q}$,
and for each $n\in{\mathbb N}$, let $U_n=(W_n\setminus\{p\})\times\{n\}$.  Then
${\mathcal U}=\{U_n:n\in{\mathbb N}\}$ is a family of pairwise disjoint
nonempty open subsets of $X\times Y$.

Finally, assume $K$ is any compact subset of $X\times Y$. To complete the
 proof, we 
need to show that $K\cap U_n=\emptyset$ for some $n\in{\mathbb N}$. 
Since the projection ${\rm pr}_X (K)$ is a compact subset of $\beta{\mathbb Q}$
not containing $p$, the set $T=\beta{\mathbb Q}\setminus{\rm pr}_X (K)$ is an
open neighborhood of $p$
in $\beta{\mathbb Q}$
 which is disjoint from ${\rm pr}_X (K)$.  As $\mathcal W$ is a $\pi$-base
for $\beta{\mathbb Q}$, $T$ must contain some $W_n$.  Hence,
$K\cap U_n\subset K\cap((W_n\setminus\{p\})\times Y)\subset
 K\cap(T\times Y)=\emptyset$.

\enddemo
\smallskip

We turn next to producing a negative answer to Question 1.4, and we shall
do this by finding a \cl\ (actually, another cellular-compact) space $X$ having
the form $X=K\setminus\{p\}$, where $K$ is compact, such that
the product space $X\times(\kappa+1)$, where $\kappa=\pi\chi(p,X)$,
 is not \cl.

\begin{Lemma} There exists a compact space $K$ having a 
 $P$-point $p$ such that $K$ has cellularity $\mathfrak c<\pi\chi(p,K)$.
\end{Lemma}

\demo
Consider the $P$-space $Y$ obtained by endowing the product space
 $2^{{\mathfrak c}^+}$ with the $G_\delta$-topology. Let $p$ be the constant
0 function in $Y$ and $K=\beta Y$.

We verify that the $P$-point $p$ of $K$ has
 $\pi\chi(p,K)={\mathfrak c}^+$. Since $K$ is an extension space of $Y$ which is regular, it will be enough to check that
 $\pi\chi(p,Y)={\mathfrak c}^+$.

Let $\mathcal A$ be the family of all countable subsets of ${\mathfrak c}^+$.  For each $A\in\mathcal A$ let
$N_A=\{y\in Y:y(\alpha)=0 \rm{\ for\ all\ } \alpha\in A\}$.  Then $\{N_A:A\in{\mathcal A}\}$ is a local base
for $p$ in $Y$, so $\pi\chi(p,Y)\leq\pi(p,Y)\leq|\mathcal A|={\mathfrak c}^+$.  Assume next that 
$\mathcal V$ is an arbitrary family of basic clopen $G_\delta$-subsets of $Y$ such that $|\mathcal V|\leq{\mathfrak c}$.
For each $V\in\mathcal V$, let $D_V$ be the countable set of restricted coordinates of $V$, and let
$\mathcal D=\{D:V_D\in\mathcal V\}$.
  Then ${\mathcal V}$ cannot be a local $\pi$-base for $p$ in $Y$, since for any
$\alpha\in{{\mathfrak c}^+}\setminus\bigcup{\mathcal D}$, ${\pi_\alpha}^{-1}(0)$ is a neighborhood of
$p$ in $Y$ which contains no member of $\mathcal V$.  Thus, $\pi\chi(p,Y)\geq{{\mathfrak c}^+}$.

Suppose next that $\{f_\alpha:\alpha<{\mathfrak c}^+\}$ are countable partial functions from ${\mathfrak c}^+$ to 2 coding basic clopen $G_\delta$-sets 
$\{[f_\alpha]:\alpha<{\mathfrak c}^+\}$ of $Y$.  Let this family be an element of
 $M\prec H({\mathfrak c}^+)$, where $M$ has cardinality $\mathfrak c$ and
$M^{\omega}\subset M$. Let
 $M\cap\,{\mathfrak c}^+\subset\delta\in{\mathfrak c}^+$,
and let $f=f_{\delta}\upharpoonright M$. 
 Since the domain of $f$ is a countable subset of $M$, $f$ is in $M$ by our choice of $M$.
Thus there is an
 $\alpha\in M$ such that $f\subset f_{\alpha}$.  It follows that
$f_{\alpha}\cup f_{\delta}$ is a function, and therefore that
$[f_{\alpha}]\cap[f_{\delta}]$ is not empty. This proves that $K$ has
cellularity equal to $\mathfrak c<\pi\chi(p,K)$.
\enddemo

\begin{Theorem}
Let $X$ be any space having
the form $X=K\setminus\{p\}$, where $K$ is a compact space, 
$p$ is a 
 $P$-point of $K$, and $K$ has cellularity $\mathfrak c<\pi\chi(p,K)=\kappa$.
Topologize the set $\kappa+1$ so that it is
the one-point compactification of the discrete space $\kappa$.
Then the space $X$ is cellular-compact, but
the product space $X\times(\kappa+1)$ is not \cl.
\end{Theorem}

\demo
It follows from Corollary 2.1 (or Corollary 3.17 of \cite{TW}) and the cellularity inequality satisfied by $K$ 
that the space $X$ is cellular-compact.

Now fix a local $\pi$-base $\{U_\alpha:\alpha<\kappa\}$ for $p$ in $K$.
We may assume that $p\notin U_\alpha$ for all $\alpha<\kappa$.  For each
$\alpha\in\kappa$,
 the set $W_{\alpha}={U_{\alpha}}\times\{\alpha\}$ is open
in $X\times(\kappa+1)$, and  
${\mathcal W}=\{W_\alpha:\alpha<\kappa\}$
is a pairwise disjoint family of nonempty open sets.

Assume that $Y$ is a subspace of $X\times(\kappa+1)$ such that
 $Y\cap W_\alpha\neq\emptyset$ for all $\alpha\in\kappa$.
We prove that $Y$ cannot be Lindel\" of.
 For each $y\in Y$,
choose a neighborhood  $W_y$ of $y$ and subset $V_y$ of $X$ such that 
$W_y\subset V_y\times(\kappa+1)$ and
the point $p\notin\overline{V_y}$.  If there is a countable subcover
 $\{W_{y_n}:n\in\omega\}$ of $Y$, then since $p$ is a $P$-point, we may choose $\alpha<\kappa$ such that
$U_\alpha\cap V_{y_n}=\emptyset$ for all $n$. This implies that $Y$ is then disjoint
from $W_\alpha$.
\enddemo

\begin{Remark}
For any space $X$ and $P$-point $p$ of $X$, if $p$ has character $\aleph_1$,
then there is a pairwise disjoint local $\pi$-base at $p$, and hence
one can show the space $X\setminus\{p\}$ cannot be \cl.
\end{Remark}

One corollary to Theorems 2.2 and 2.4 which illustrates a difference 
between properties of the spaces studied here and those in \cite{RT} is
the following.

\begin{Corollary} The property cellular-Lindel\" of (cellular-compact) is
not an inverse invariant of open perfect surjections.
\end{Corollary}

\section{Seeking a counterexample of countable spread}
\smallskip

We consider next some properties spaces $X$ and $Y$ would need to have in 
order to be counterexamples to the assertion that if $X$ is
 \cl\ and  $Y$ is of countable spread, then $X\times Y$ is \cl.

As noted earlier, neither space can be separable. In addition,
the following theorem, a corollary to one
 due to B.\ \v Sapirovski\u \i\ in \cite{S}, shows the
 space $Y$ must have a dense hereditarily
Lindel\" of subspace.

\begin{Theorem}{\rm(\cite{H}, Proposition 5.6, or \cite{S})}
 Every non-separable space  of countable spread has a
dense subspace that is hereditarily Lindel\" of.
\end{Theorem}

A tool often found useful in producing counterexamples is
 the {\it Alexandroff double ${\mathbb D}(X)$ of a topological space
 $X$}.  We recall that ${\mathbb D}(X)= X\times\{0,1\}$,
 topologized so that a subset $T$ of
 ${\mathbb D}(X)$ is defined to be open if and only if for each point
$(x,0)\in T$, there is an open subset $U$ of the space $X$
such that $(x,0)\in \left((U\times\{0,1\})\setminus\{(x,1)\}\right)\subset T$.
Three well-known, easily verified properties are:  ${\mathbb D}(X)$ is
Lindel\" of if and only if $X$ is Lindel\" of; and since $X\times\{1\}$ is
an open discrete subspace, ${\mathbb D}(X)$ is separable (hereditarily
Lindel\" of) if and only
if $X$ is countable.

\begin{Theorem} 
Suppose $X$ and $Y$ are regular spaces
such that
 ${\mathbb D}(X)\times Y$ is \cl,
$Y$ has $\pi$-weight $\pi(Y)=\kappa>\aleph_0$,
 and  
 for every nonempty open subset $T$ of $X$,
 ${\rm d}(T)=\pi(T)=\kappa$. Then ${\mathbb D}(X)\times Y$ has a
Lindel\" of subspace which is dense in $(X\times\{0\})\times Y$.
\end{Theorem}

\demo
A recursive construction shows that
 there is a family $\{D_\alpha:\alpha<\kappa\}$ of pairwise
disjoint dense subsets of $X$, each having cardinality $\kappa$.

Let $\{U_\alpha:\alpha\in\kappa\}$ enumerate a $\pi$-base for $Y$.
Set $W_{\alpha,d}=\{(d,1)\}\times U_\alpha$ for all $\alpha<\kappa$ and
 $d\in D_\alpha$, and $\mathcal W=\{W_{\alpha,d}:\alpha\in\kappa
 {\rm\ and\  } d\in D_\alpha\}$. Then $\mathcal W$ is a cellular family, so
assume that $L$ is a Lindel\" of subspace of
 ${\mathbb D}(X)\times Y$ such that $L\cap W\neq\emptyset$ for
all $W\in\mathcal W$.  We shall prove that $L$ is dense
in $(X\times\{0\})\times Y$.

Fix any nonempty regular closed sets $B\subset X$ and
$R\subset Y$. We show that
 $L\cap((B\times\{0\})\times R)\neq\emptyset$.

Since $\{U_\alpha:\alpha\in\kappa\}$ is a $\pi$-base for $Y$, there
exists $\beta<\kappa$ such that $U_\beta\subset{\rm int}(R)$. As
$D_\beta$ is a dense subset of $X$ and $d({\rm int}(B))=\kappa$,
it follows that
 $L\cap ((B\times\{0,1\})\times R)$ 
meets $W_{\beta,d}$ for $\kappa$-many points $d\in D_\beta$.
 For each such $d$,
 pick $((d,1),y_d)\in L \cap W_{\beta,d}$, and note that
$((d,1),y_d)\in L\cap ((B\times\{0,1\})\times R)$, Moreover,
$ L\cap ((B\times\{0,1\})\times R)$ is Lindel\" of.
As this subset of points
$((d,1),y_d)\in   L\cap((B\times\{0,1\})\times R)$ has cardinality
 $\kappa>\aleph_0$,
 it  must have a limit point in
 $L\cap((B\times\{0\})\times R)$, for
 every point of $L\cap((B\times\{1\})\times R)$ has a neighborhood
meeting that subset in at most a single point.

\enddemo
\smallskip

Based on the preceding results, we seek as counterexamples spaces $X$ and $Y$,
 each nowhere separable and of $\pi$-weight $\aleph_1$, so that $X$ is
 Lindel\" of, $Y$ is hereditarily Lindel\" of, and ${\mathbb D}(X)\times Y$
 has no Lindel\" of subspace which is dense in $(X\times\{0\})\times Y$.
\smallskip

\section{Consistent product space counterexamples}

Let $S$ be a full-branching Souslin tree (a subtree of $2^{<\omega_1}$).
In this section we show how to get $X$ and $Y$ from $S$. We begin by
giving some notation and terminology that will be used.

Let
$s, t\in S$ and  $A\subset S$. Then $s\leq t$ is assigned its usual meaning
 (as on p. 68 of \cite{K}),
namely, $s\leq t$ if and only if $s\subset t$ if and only if the sequence
 $t$ extends $s$, and likewise $s\lor t$, $s\land t$, and $s\perp t$
have their usual meanings with respect to the relation $\leq$. In addition,
 $o(s)\in\omega_1$ denotes the domain of $s$,
 $[s]$ denotes $\{u\in S: s\leq u\}$, and $A$ is said to be
 {\it cofinal above $s$} if for each $u\in[s]$, $A\cap[u]\neq\emptyset$.
The set $A$ is said to be {\it cofinal in $S$} if $A$ is cofinal above
$s$ for every $s\in S$.
If $o(s)=\xi$ and $i, j\in\{0,1\}$, then $s^\frown i$ denotes the
sequence $s\cup\{(\xi,i)\}$ whose domain is $\xi+1$, and similarly,
$s^\frown ij$ denotes $(s^\frown i)^\frown j$. In case $s\neq t$,
we define $\Delta(s,t)={\rm min}\{\alpha\in o(s)\cup o(t):s(\alpha)\neq t(\alpha)\}$
(using the convention that $s(\alpha)\neq t(\alpha)$ if $\alpha\in (o(s)\setminus o(t))\cup(o(t)\setminus o(s))$), and
denote by $\prec$ the following order: $s\prec t$ if and only if
$s(\Delta(s,t))=0$ or $t(\Delta(s,t))=1$.

The reader familiar with the procedures M.E. Rudin developed on pp. 1116--1118 of \cite{Ru} may notice that
the relation $\prec$ is an example of a general type of relation presented and studied in \cite{Ru} for the purpose of demonstrating how to use a ``normalized and pruned" Souslin tree and Dedekind completion to produce a Souslin line.  We include for later reference statements and proof outlines illustrating the properties of $\prec$ needed for our examples. Some of these properties were derived or pointed out in \cite{Ru}. 

\begin{Claim}
  The relation $\prec$ is a dense total order on $S$.
\end{Claim}

\demo
 First we prove $\prec$ is transitive. Let $s, t$, and $ u\in S$
satisfy $s\prec t$ and $t\prec u$, and let $\alpha_0=\Delta(s,t)$ and
$\beta_0=\Delta(t,u)$. We consider three cases.

Suppose $\alpha_0=\beta_0$.  Then 
$s\upharpoonright\alpha_0=t\upharpoonright\alpha_0=u\upharpoonright\alpha_0$,
and the only possible values for $s$, $t$, and $u$ at $\alpha_0$ in this case
are $s(\alpha_0)=0$, $t(\alpha_0)$ is not defined, and $u(\alpha_0)=1$, which
together
imply $s\prec u$.

Suppose $\alpha_0<\beta_0$. Then
 $s\upharpoonright\alpha_0=t\upharpoonright\alpha_0=u\upharpoonright\alpha_0$.
 If
$s(\alpha_0)=0$ or is undefined, then $t(\alpha_0)=1$,
 and $t\upharpoonright\beta_0=u\upharpoonright\beta_0$ implies
$u(\alpha_0)=1$, and so $s\prec u$.

Assume $\beta_0<\alpha_0$.
  Then $s\upharpoonright\beta_0=t\upharpoonright\beta_0$,
$\beta_0\in o(t)$,
 and 
$s(\beta_0)=t(\beta_0)$.
Since also $t\prec u$, we must have $t(\beta_0)=0$.
  Then $s(\beta_0)=0$, $s\upharpoonright\beta_0=u\upharpoonright\beta_0$,
 and since $t\prec u$, either
$u(\beta_0)=1$ or $u(\beta_0)$ is not defined.
 Hence $s(\beta_0)=0\neq u(\beta_0)$,
and  $s\prec u$.

Obviously the partial order $\prec$ is a total order.  To prove it is a
dense order, let $s$ and $t$ be as above. We find a sequence $m\in S$ such
that $s\prec m\prec t$, depending on which of the three possibilities occurs.

If $s(\alpha_0)=0$ and $t(\alpha_0)=1$, let $m=s\upharpoonright\alpha_0$.

Suppose $s(\alpha_0)=0$ and $t(\alpha_0)$ is not defined.  Find the
 first $\beta\geq\alpha_0+1$ where $s(\beta)=0$, and set 
$m=(s\upharpoonright\beta)^\frown1$;
 and if no such
$\beta$ exists, set $m=s^\frown1$.

Assume $s(\alpha_0)$ is not defined and $t(\alpha_0)=1$. Find the first
$\beta\geq\alpha_0+1$ where $t(\beta)=1$, and set $m=(t\upharpoonright\beta)^\frown0$; and if
no such $\beta$ exists, let $m=t^\frown0$. 
\enddemo
\smallskip

Before stating our next claim, some additional notation needed is, for
any $t\in S$: let $t^\dagger=t$ if $t$ is on a limit level of $S$; and
let $t^\dagger$ denote the ``other'' $<$-successor of $t\upharpoonright\alpha$
in $S$ if $o(t)=\alpha+1$.

\begin{Claim}
Let $s,t\in S$ with $o(s)+1<o(t)$. {\rm(i)} If $s\prec t$ then 
 $s\prec([t]\cup[t^\dagger])$, and {\rm(ii)} if  $t\prec s$ then
 $([t]\cup[t^\dagger])\prec s$.
\end{Claim}

\demo
Let $s, t\in S$, and assume the hypothesis of (i) holds. Then whether 
 $s(\Delta(s,t))$ is not defined or equals 0, since
$o(s)+1<o(t)$, it follows that 
 $t(\Delta(s,t))=1$, $t^\dagger(\Delta(s,t))=1$, and
$\Delta(s,t)=\Delta(s,t^\dagger)$. Thus,
for any $u\in S$, if  $u\supset t$ or $u\supset t^\dagger$ then
$\Delta(s,u)=\Delta(s,t)$,  
$u(\Delta(s,u))=t(\Delta(s,t))=1$, and hence $s\prec u$.
The statement  (ii) is proved similarly.
\enddemo
\smallskip

For the remainder of this section, $S$ is endowed with the {\it $\prec$-order
topology}, and the {\it open interval notation $(s,t)$}
will refer to $\prec$, where $s\prec t$ (except when it is used to denote a
point in the product space $S\times S$).  Thus, {\it topological terms referring to $S$} should be understood to be referring to the
topology induced on $S$ by $\prec$ (but {\it not} to the one induced on it by the tree order $<$).

\begin{Claim}
For each $s\in S$, $[s]$ is a clopen subset of $S$, and
the family $\{[s]: s\in S\}$ is a $\pi$-base for the topology on $S$.
\end{Claim}

\demo
First, we observe that each $[s]$ is an open set: for any $t\geq s$,
 $t^\frown0\prec t\prec t^\frown1$, and so $I= (t^\frown0,t^\frown1)$ is
an open interval containing $t$, and one can check that
$s\leq t\leq u$ for any $u\in I$, and hence $I\subset[s]$.

To see that $[s]$ is a closed set, consider any point $t\in S$ such that
$s\nleq t$. Let $\alpha=\Delta(s,t)$. Then $\alpha\in o(s)\cap o(t)$
(if $s\perp t$) or
$\alpha\in o(s)\setminus o(t)$.  
 If $\alpha\in o(t)$, then
$[t]$ is an open neighborhood of $t$ disjoint
with $[s]$.  If $\alpha\notin o(t)$, then $[t^\frown(1-s(\alpha))]$ is
an open neighborhood of $t$ disjoint with $[s]$.

Next, let us prove that the family
 $\{(v_1,v_2): v_1\prec v_2{\rm\ and\ }v_1\perp v_2\}$ is a $\pi$-base.
Fix any $s\prec t$ with $s\neq t$, and we find $(v_1,v_2)\subset(s,t)$ from
this $\pi$-base.  If $s\perp t$, then $v_1=s$ and $v_2=t$ suffice. So
assume $s\not\perp t$, and we consider the remaining  possibilities.

Assume $s\leq t$, i.e., $s\subset t$. Then $s(\Delta(s,t))$ is not defined
and $t(\Delta(s,t))=1$, since $s\prec t$. Thus $s^\frown1\subset t$, and we
define $v_1=t^\frown00$ and $v_2=t^\frown01$. 

Assume $t\leq s$, i.e., $t\subset s$. Then $s(\Delta(s,t))=0$ and
 $t(\Delta(s,t))$ is not defined, since $s\prec t$.
 Hence $t^\frown0\subset s$, and we
define $v_1=s^\frown10$ and $v_2=s^\frown11$. 

In either of these definitions, it follows easily that $v_1\perp v_2$ and
  $s\prec v_1\prec v_2\prec t$.

To finish  the proof,
 we just  prove that for each $v_1\prec v_2$ with
$v_1\perp v_2$, there is some $s$ with $v_1\prec[s]\prec v_2$.
Evidently $\Delta(v_1,v_2)\in o(v_1)\cap o(v_2)$, so
$v_1(\Delta(v_1,v_2))=0<1=v_2(\Delta(v_1,v_2))$. Let $s={v_1}^\frown1$. Hence
$v_1\prec[s]$.  Furthermore, if $s\subset t$, then $t(\Delta(v_1,v_2))=
s(\Delta(v_1,v_2))=v_1(\Delta(v_1,v_2))=0$,
and $t\upharpoonright\Delta(v_1,v_2)=
s\upharpoonright\Delta(v_1,v_2)=
v_1\upharpoonright\Delta(v_1,v_2)=
v_2\upharpoonright\Delta(v_1,v_2)$. 
Thus $t\prec v_2$, which shows $[s]\prec v_2$.
\enddemo

\begin{Claim}
The space $S$ is hereditarily Lindel\" of.
\end{Claim}

\demo
Since $S$ is a linearly ordered space (with respect to the $\prec$-order), to prove that it is hereditarily Lindel\" of, it suffices by Theorem 2.2 in \cite{LB} to prove that it satisfies the countable chain condition ($ccc$). It was shown in \cite{Ru} that the Dedekind completion of any such space is $ccc$.  Using the same method, we apply Claim 4.3 here and observe that if $\mathcal A=\{[s_\alpha]:\alpha\in A\}$ is a pairwise disjoint listing of the members of a family of elements of the $\pi$-base for $S$, then $\{s_\alpha:\alpha\in A\}$ is an antichain in  the Souslin tree $(S,\leq)$, and so $\mathcal A$ must be countable. 
\enddemo

\begin{Claim}
Let $\{(u_\xi,v_\xi): \xi<\omega_1\}$
be a subset of $S\times S$ such that for all $\xi<\omega_1$,
there is a $t_\xi$ with $o(t_\xi)\geq\xi,\ t_\xi\leq u_\xi$, and
 $t_\xi\neq t_\xi^\dagger\leq v_\xi$ (hence every $o(t_\xi)$ is a successor).
  Then this set co-countably converges
to the diagonal in $S\times S$.
\end{Claim}

\demo
Any open neighborhood of the diagonal $\Delta$ in $S\times S$ will contain a
set of the form W=$\bigcup\{(s^-,s^+)\times(s^-,s^+):s\in S\}$, where for all
 $s\in S$, $s^-\prec s\prec s^+$.  Since $S$ (and hence $\Delta$) are
 Lindel\" of
\ and $S$ is
a Souslin tree, there is a $\delta<\omega_1$ such that the collection
$\{(s^-,s^+):s\in S{\rm\ and\ }o(s)<\delta\}$ is a countable cover of $S$.
Let $\delta\leq\beta<\omega_1$ be large enough so that
 max$\{o(s^-),o(s^+)\}<\beta$ for all $s\in S$ with $o(s)<\delta$. Choose
any $\xi<\omega_1$ such that $\beta+1<\xi$, and note that $\beta+1<o(t_\xi)$.
Since $t_\xi\neq t_\xi^\dagger$, we can let ${\overline t}_\xi$ be the
member of $S$ so that  $t_\xi$ and $t_\xi^\dagger$ are its immediate
successors.  Choose $s\in S$ with $o(s)<\delta$ such that
 $s^-\prec{\overline t}_\xi\prec s^+$. Since
 $o({\overline t}_\xi)>\beta\geq{\rm max}\{o(s^-),o(s^+)\}+1$, it follows from
Claim 4.2 that $[{\overline t}_\xi]\subset(s^-,s^+)$,
  i.e., $s^-\prec[{\overline t}_\xi]$ and $[{\overline t}_\xi]\prec s^+$. 
 This proves that
$(u_\xi,v_\xi)\in W$.
\enddemo

\medskip

Now we are ready to construct the example. Let $S^+$ denote the set of all
$s\in S$ such that $o(s)$ is a successor,  $U$ denote the set of all
 $s\in S^+$ such that for $o(s)=\alpha+1$, $s(\alpha)=0$, and
$V=S^+\setminus U$. Then note that
$V=\{u^\dagger: u\in U\}$, each of $U$ and $V$ is cofinal in $S$ (which means, as defined at the beginning of this section, with respect to the tree order $\leq$),
 and thus $U$ and $V$ are disjoint dense subsets of
$S$.  We let $X={\mathbb D}(U)$, the Alexandroff double of $U$.

\begin{Theorem} The space $X={\mathbb D}(U)$ is Lindel\" of,
 and the space $V$ is hereditarily
Lindel\" of, but $X\times V$ is not \cl..
\end{Theorem}

\demo
It suffices to prove that $X\times V$ is not \cl.

Let $\{t_\xi:\xi\in\omega_1\}$ be an enumeration of a subset of $U$ such that
for each $\xi<\omega_1$, $o(t_\xi)\geq\xi$.  For each $\xi\in\omega_1$,
choose any $u_\xi\in U$ and $w_\xi\in V$ such that $t_\xi\leq u_\xi$,
$u_\xi\neq u_\eta$ for any $\eta<\xi$, and
$t_\xi^\dagger<w_\xi$. Also choose $w_{\xi}^-\prec w_\xi\prec w_{\xi}^+$
in $[t_{\xi}^\dagger]$, and note that either $t_\xi\prec w_{\xi}^-$ or
 $w_{\xi}^+\prec t_\xi$.  Now, for each $\xi<\omega_1$, the set
$U_\xi=\{(u_\xi,1)\}\times(V\cap (w_{\xi}^-,w_{\xi}^+))$
 is a nonempty open
subset of $X\times V$, and $\{U_\xi:\xi<\omega_1\}$ is pairwise disjoint.
 
 Let $Y\subset X\times V$, and assume that $Y\cap U_\xi$
is nonempty for all $\xi<\omega_1$. We prove that $Y$ is not Lindel\" of by
proving that it has an uncountable subset $R$ with no complete accumulation 
point in
$X\times V$.
  For each $\xi\in\omega_1$ choose
 $v_\xi\in V\cap (w_{\xi}^-,w_{\xi}^+)$ so that
 $((u_\xi,1),v_\xi)\in U_\xi\cap Y$, and let
$R=\{(u_\xi,1),v_\xi):\xi<\omega_1\}$.
Note that by construction, each
 $v_\xi\in[t_{\xi}^\dagger]$.

  Fix any $(u,v)\in U\times V$. Since $\{(u,1)\}\times V$ is a neighborhood
of the point $((u,1),v)$  which intersects $R$ in at most one point,
$((u,1),v)$ is not a limit point of $R$. Let us
 prove that neither is
$((u,0),v)$  a complete accumulation point of
$R$.
 Choose $u^-\prec u\prec u^+$ and
$v^-\prec v\prec v^+$ (from $S$) so that $(u^-,u^+)\cap (v^-,v^+)$ is
empty, and so that the closure of $(u^-,u^+)\times (v^-,v^+)$ is disjoint 
from the diagonal in $S\times S$.
  By Claim 4.5, we may choose $\eta<\omega_1$ so that
 $(u_\xi,v_\xi)\notin((u^-,u^+)\times(v^-,v^+))$ for all $\eta<\xi<\omega_1$.
It follows easily that
 $((u_\xi,1),v_\xi)\notin((u^-,u^+)\times\{0,1\})\times(v^-,v^+)$ for
 all $\eta<\xi<\omega_1$, and that completes the proof.
\enddemo
\smallskip

\begin{Remark} Another way to prove Theorem 4.6 would be to prove that
there is no Lindel\" of subspace of ${\mathbb D}(U)\times V$ which
is dense in $(U\times\{0\})\times V$, and then appeal to Theorem 3.2.
\end{Remark}

\section{ZFC product space counterexamples}
In this section we show that  Moore's L space
contains  a pair of nowhere separable hL spaces whose product has no
dense Lindel\" of subspace.

Moore defines his L space after developing the key properties
of 
a very special combinatorial 
object $\mathop{Osc}$. The object
$\mathop{Osc}$ is a function from
the ordered pairs $\alpha<\beta<\omega_1$ into the
finite subsets of $\alpha$. Similarly, the function
 $\mathop{osc}(\alpha,\beta)$ is defined as
the cardinality of the set $\mathop{Osc}(\alpha,\beta)$.
  These  definitions depend on  the choice of
  a family $\vec{\mathcal C}=
  \langle C_\alpha :\alpha\in\omega_1\rangle$
  where
  
  \begin{Definition}
    A family $\vec{\mathcal C}
    =  \langle C_\alpha :\alpha\in\omega_1\rangle$ is
    a C-sequence on $\omega_1$ provided\label{c-sequence}
    \begin{enumerate}
      \item $C_0=\{0\}$,
    \item for each $0<\alpha<\omega_1$,
      $0\in C_\alpha$ is a cofinal subset of $\alpha$
      (in case $\alpha=\beta+1$, this just means $\beta\in C_\alpha$),
      \item for all $\gamma<\alpha$, $C_\alpha\cap \gamma$ is finite.
    \end{enumerate}
  \end{Definition}

We will be using properties of the function
$\mathop{osc}$ and Moore's L space from both papers
\cite{MooreL} and \cite{Lgroup}. In each paper, it is stated that any
choice  of a C-sequence will suffice. We mention this here because
in order to establish our main result, we will impose
an additional restriction on the C-sequence that we use.

Here is the definition of Moore's L space from \cite{MooreL}.
First of all, $\mathbb T$ denotes the unit circle in the complex plane,
and for $z=e^{i\,r}\in \mathbb T$ and $k\in \mathbb N$, of course
$z^k = e^{i(kr)}$ is also an element of $\mathbb T$.
The topology on $\mathbb T$ is the one inherited as a subspace
of the complex plane. We let ${1}\in\mathbb T$ denote
the element $(1,0)$.

\begin{Definition} Let $\{ z_\alpha : \alpha\in \omega_1\}$ be a
  rationally independent set. For all $\beta\in\omega_1$,
  define $w_\beta\in \mathbb T^{\omega_1}$ with the formula\\
  \centerline{\( w_\beta(\xi) = \begin{cases}
    z_\xi^{\mathop{osc}(\xi,\beta)+1} & \xi<\beta\\
      {1} & \beta\leq\xi.\end{cases}\)}
\end{Definition}

\begin{Theorem}[\cite{MooreL}, {7.11}]
  For all uncountable $X\subset\omega_1$,
  $\mathcal L_X = \{ w_\beta\restriction
   X : \beta\in X\}$ is hereditarily Lindel\" of.
\end{Theorem}

For each $X\subset\omega_1$,
 $\mathcal L[X] = \{ w_\beta : \beta\in X\}$.
It is also noted in \cite{MooreL} that the closures
of countable subsets of $\mathcal L=\mathcal L[{\omega_1}]$ are countable,
and therefore we have the following

\begin{Lemma}
  For all uncountable $X\subset\omega_1$,\label{ctble} 
  $\mathcal L[X]$ is a nowhere separable hereditarily Lindel\" of space.
\end{Lemma}

 For convenience we make
some assumptions about the family $\{ z_\alpha :\alpha\in\omega_1\}$.
We choose a rationally independent set 
$\{ r_\alpha : \alpha <\omega_1\}$ of reals in the interval
$(0,\frac1{2\pi})$, and for each $\alpha<\omega_1$ we let
$z_\alpha $ be the point in $\mathbb T$ obtained
by rotating ${1}$ by $2\pi r_\alpha$ radians.
For convenience we adopt the following metric on $\mathbb T$.
For points  $z,w$  of
$\mathbb{T}$ we let  $\rho(z,w)$ be
the smallest value $0\leq r<2\pi$ so that
 $w\in \{ e^{-2\pi r\,i} z,e^{2\pi r\,i} z\}$.
  In other words, $\rho(z,w)$ is $\frac{1}{2\pi}$
times the arc length within $\mathbb{T}$ between
$z$ and $w$. Note that (since $\mathbb{T}$ is a group)
for any $u\in \mathbb{T}$, 
 $\rho(z,w) = \rho(u\cdot z,u\cdot w)$. 

 For any point $w\in \mathbb{T}^{\omega_1}$,
let $U[w;F,\epsilon]$ (for finite $F\subset\omega_1$ and
real $\epsilon>0$) denote the basic open set consisting of all
points $w'\in \mathbb{T}^{\omega_1}$ satisfying
that $\rho(w'(\xi),w(\xi))<\epsilon$ for all $\xi\in F$.

Let $\Lambda$ denote the set of 
limit ordinals in $\omega_1$, and then let
$\Lambda'$ denote those members $\delta$
of $\Lambda$
such that $\Lambda\cap \delta$ is cofinal in $\delta$
(i.e., $\Lambda'$
is the set of limit points of $\Lambda$).

\begin{Lemma} There is a choice of the C-sequence
  $\vec{\mathcal C}=\langle C_\alpha :\alpha\in\omega_1\rangle$ so
  that there
are disjoint uncountable\label{specialpairs}
 subsets
 $X$ and $Y$ of $\omega_1$ satisfying that
 for all $\delta\in \Lambda'$ and
 $\delta<\alpha\in X$ 
 (respectively $Y$)
and $\alpha < \beta\in Y$ (respectively $X$),
$|\mathop{Osc}(\alpha,\beta)\cap \delta|>1$. 
\end{Lemma}

We will prove Lemma \ref{specialpairs} later.   For now
we fix the pair of uncountable sets $X$ and $Y$.
One  of the key properties
of $\mathop{osc}$ 
is proven in Lemma 4.4 of \cite{MooreL}. This
is improved in Lemma 8 of \cite{Lgroup} by Peng
and Wu. We will just need a minor variation of
a special case of this Lemma 8 that we 
record here. The set of two-element subsets of $\omega_1$ is
denoted by $[\omega_1]^2$.  For $a\in [\omega_1]^2$, we let
$a(0)=\min(a)$ and $a(1)=\max(a)$. For $a,b\in[\omega_1]^2$, we
let $a<b$ denote the relation that $a(1)<b(0)$. We will
apply this Lemma to pairs $a\in [\omega_1]^2$ that have
an element in each of $X$ and $Y$ from Lemma \ref{specialpairs}.
It follows that the integer $c$ in the statement of the next
lemma would then be at least $2$.

\begin{Lemma}
  Suppose that $\mathcal A\subset [\omega_1]^2$ is\label{shrink}
 an uncountable
  family of pairwise disjoint sets. 
 Then there are an uncountable collection $\mathcal A'\subset \mathcal
 A$ ,  an integer $c$, and a $\delta\in\Lambda'$
 satisfying that for any $a < b$, both in $\mathcal A'$,
 $|\mathop{Osc}(a(0),a(1))\cap \delta|=c$ and\\
 \centerline{$
  \mathop{osc}(a(0),b(0))+ c-1\leq \mathop{osc}(a(0),b(1))
  \leq
      \mathop{osc}(a(0),b(0))+ c.$}
\end{Lemma}

We defer the proofs of Lemma \ref{specialpairs} and
 Lemma \ref{shrink} until after 
 we prove that the main result is a consequence.

\begin{Theorem} If $X$ and $Y$ are the 
disjoint uncountable subsets as in Lemma \ref{shrink},
 then $\mathcal L [X]\times \mathcal L [Y]$ does
  not contain a dense Lindel\" of subset.
\end{Theorem}

\begin{proof} 
Assume that $D\subset X\times Y$ and
  that $\tilde D = \{ (w_\beta ,w_\gamma ) : 
 (\beta,\gamma)\in D\}$ is dense in $\mathcal L [X]\times
  \mathcal L [Y]$. 
  We will  produce an uncountable subset
  of $\tilde D$  that has no complete accumulation point
  in $\mathcal L [X]\times \mathcal L [Y]$.
  \medskip
  
By Lemma \ref{ctble}, we have
  that, for all $\delta\in\omega_1$,
 $$\tilde D_{\delta} = 
  \{ (w_\beta ,w_\gamma ) :  (\beta,\gamma)\in D\cap
  (\delta\times\omega_1\cup\omega_1\times \delta)\}$$
  is not  dense.
 Therefore we may choose
  an uncountable subset $D_1\subset D$ that has
  the property that if $d_1,d_2\in D_1$ are distinct,
  then there is a $\delta\in \omega_1$ such that (wlog)
  $d_1\in \delta\times\delta$ and
  both coordinates of $d_2$ are greater than $\delta$. By passing
  to a subset of $D_1$ we can assume that (wlog)
  for each $d\in D_1$,
the first
coordinate  of $d$ ($d(0)$ in $X$) is less than the second coordinate
of $d$ ($d(1)$ which is in $Y$). 
Let $\mathcal A$ be the uncountable set of
disjoint pairs $\{ \{d(0),d(1) \} : d\in D_1\}$.
By applying Lemma \ref{shrink}, we may choose an
integer $c$ and an
uncountable $D_2\subset D_1$ such
that for all $a,b\in D_2$ with $a(1)<b(0)$,
 we have that
 $\mathop{osc}(a(0), b(0))+ c -1\leq
 \mathop{osc}(a(0) ,b(1) ) \leq
 \mathop{osc}(a(0), b(0))+ c 
 $.
 By Lemma \ref{specialpairs}, the value of $c-1$ is positive.

 For each real $r$, we  let
 $[r]_{2\pi} $ denote the value $r-2\pi\ell$,
 where $2\pi\ell\leq r<2\pi(\ell+1)$,
 i.e., $e^{i\,r}= e^{i\,[r]_{2\pi}}$.
 Note that for $a < d$, both in $D_2$
 and $k=\mathop{osc}(a(0),d(0))+1$, the value
  of 
  $\rho(w_{d(0)} (a(0)), w_{d(1)} (a(0))) $ equals one of
  $$
  \rho( z_{a(0)}^k,  z_{a(0)}^{k+c})
  =   \rho( {1}, z_{a(0)}^{c})
 =    [ c\,r_{a(0)}]_{2\pi}\rm{\ and}$$
  $$
  \rho(z_{a(0)}^k,  z_{a(0)}^{k+c-1})
  =   \rho({1},  z_{a(0)}^{c-1})
 =   [ (c-1)\,r_{a(0)}]_{2\pi}~~.$$

 Now we choose  three uncountable subsets of $D_2$.
 First of all, choose any pair $0<s_1 < s_2 $
 so that \begin{enumerate}
 \item  $s_1$ and $s_2$ are complete accumulation
 points of $(s_1,s_2)\cap\{ r_{d(0) }  : d\in D_2\}$,  
   \item and 
there is an $\ell$ with $\pi\ell< cs_1
< c s_2 < \pi(\ell+1),$
 \end{enumerate}
and let $s =\frac{s_1+s_2}2$. 
Note that for $s_1<r_1<r_2<s_2$, \ 
$[cr_2]_{2\pi}-[cr_1]_{2\pi} = c(r_2-r_1)$.

Choose  $0<\epsilon < \frac{s_2-s_1}{5}$ 
so that  $D_4=
\{ b\in D_2 : r_{b(0)} \in (s_1+2\epsilon,s_2-2\epsilon)\}$
is uncountable. Also let $D_3=
\{ a\in D_2 : r_{a(0)} \in (s_1,s_1+\epsilon)\}$,
and
 $D_5=
\{ b\in D_2 : r_{b(0)} \in (s_2-\epsilon,s_2)\}$. We note
that each of $D_3,D_4, D_5$ is uncountable.

For any 
$a\in D_i$, $a<b\in D_j$ ($3\leq i\neq j\leq 5$), 
 $b<d\in D_2$, and $c'\in \{c-1,c\}$,
we have
that 
$$ |\, [c'\,r_{a(0)}]_{2\pi}
-  [c'\,r_{b(0)}]_{2\pi}|=
 c' |r_{a(0)}-r_{b(0)}| >  \epsilon~~~.$$
 This implies that if 
$a_3<a_4<a_5 <d$ with $a_i\in D_i$ and $d\in D_2$,
then
    there is some choice $\{a,b\}\subset\{a_3,a_4,a_5\}$ so that
$$| \rho(w_{d(0)} (a(0)),w_{d(1)} (a(0))) -
   \rho(w_{d(0)} (b(0)),w_{d(1)} (b(0))) | > \epsilon~.$$

 \bigskip

  We prove that the uncountable set
  $ \{ (z_{d(0)}, z_{d(1)}) : d\in D_2\}$ 
has no complete accumulation
 point in   $\mathcal L [X]\times \mathcal L [Y]$.
 Fix any pair $(u,v)\in X\times Y$ and choose $a\in D_3$ so
 that $\max\{u,v\}<a(0)$; then choose $b,e$
 so that  $a<b\in D_4$,
 and $b<e\in D_5$.
 Use the neighborhood $U(w_u ; \{a(0),b(0),e(0)\},\epsilon/4)
 \times U(w_v ; \{a(0),b(0),e(0)\},\epsilon/4)$  for
 $(w _u,w _v)$ in $\mathcal L [X]\times \mathcal L [Y]$.
 Let $r = \rho(z_u,z_v)$ and recall that
 $w_u (\xi)={1}$ and $w_v (\xi)={1}$ for all
  $a(0)\leq \xi\in\omega_1$. 
 Now suppose  that
 $(w_{d(0)} ,w_{d(1)} )$ is in this neighborhood for some $d\in
 D$.
It follows that for each $\xi\in \{a(0),b(0),e(0)\}$,  
$\rho(w_{d(0)} (\xi),{1})<\epsilon/4$
and $\rho(w_{d(1)} (\xi),{1})<\epsilon/4$.
This implies that $\rho(w_{d(0)} (\xi), w_{d(1)} (\xi))\in
(-\epsilon/2,\epsilon/2)$
for each $\xi\in\{a(0),b(0),e(0)\}$.
This implies that $d $ is not an element
in $ D_2$  above $e$ 
since for such $d\in D_2$
there is a pair $\xi,\zeta\in \{a(0),b(0),e(0)\}$
such that
$| \rho(w_{d(0)} (\xi),w_{d(1)} (\xi)) -
\rho(w_{d(0)} (\zeta),w_{d(1)} (\zeta)) | >\epsilon$.
\end{proof}

Now we  prove Lemma \ref{specialpairs}. 

\bgroup
\def\proofname{Proof of Lemma \ref{specialpairs}:\/}

\begin{proof}
Let $\Lambda''$ be the set of $\lambda\in \Lambda'$
such that $\Lambda'\cap\lambda$ is cofinal in $\lambda$.
To prove the Lemma it will suffice to define uncountable subsets $X$ and $Y$ of
$\omega_1$ and to then  prove that for all $x\in X$ and $y,y'\in Y$ with $y<x<y',$
$$
|\mathop{Osc}(y,x)\cap\omega|>1 \ \mbox{and}\ 
     |\mathop{Osc}(x,y')\cap\omega|>1~~.$$
In fact, to accomplish this
 we will be more prescriptive in our definitions and by the end of this proof
we will have proven

\begin{Fact} For all $x\in X$ and $y,y'\in Y$ with\label{5.8} $y
  <x<y'$,\\
  \centerline{$
 \{30,200\}\subset \mathop{Osc}(y,x)$ and
 $\{10,90\}\subset \mathop{Osc}(x,y').    $}
   \end{Fact}

Of course the specific choice of $\{10,30,90,200\}$ is quite arbitrary but
the numbers do have to be some distance apart.
Our next steps are to recall the underlying definitions from \cite{MooreL} 
 related to the $\mbox{Osc}$
function and to then construct a specific
  $ C$-sequence that will work.
Our choice of the uncountable sets $X$ and $Y$ is also critical and 
are chosen so as to allow us to carefully control the values of 
$\mbox{Osc}(\delta,z)\cap\omega$ 
for $\delta\in \Lambda'$ and $z\in (X\cup Y)\setminus \delta$.

The first  definition we will need is that of the lower trace.

\begin{Definition}[\cite{MooreL}] If $\alpha\leq\beta<\omega_1$, then
  $L(\alpha,\beta)$ is defined\label{lower} recursively by

  \centerline{\( L(\alpha,\alpha) = \emptyset\),}

  \centerline{\(L(\alpha,\beta) = 
\left(
L(\alpha,
    \min(C_\beta\setminus \alpha)) \cup
     \{\max(C_\beta\cap \alpha)\}\right)\setminus \max(C_\beta\cap
     \alpha) \).}
\end{Definition}

We will see below that $\mathop{Osc}(x,y)\subset L(x,y)$,
and so we will certainly need $L(x,y)\cap \omega$ to be non-empty
(in fact, quite large). 
Since $\max(C_\beta\cap \alpha)$ is likely to be greater than $\omega$,
in order for us to be ensuring that $\mathop{Osc}(x,y)\cap \omega$ 
is non-empty
for special pairs $x,y$, we will need to be \textit{adding\/} elements to
 $\mathop{Osc}(x,y)$ as per the recursive nature of the definition.
The following is Fact 1 from \cite{MooreL}.

\begin{Fact}
  If\label{helpful}
 $\alpha \leq\beta\leq \gamma$ and $\max(L(\beta,\gamma))
   < \min(L(\alpha,\beta))$, then\\
   \centerline{$ L(\alpha,\gamma) = L(\alpha,\beta)\cup
      L(\beta,\gamma)$.}
\end{Fact}

The next definition that we need is the following recursively defined
function $\varrho_1$ of Todor\v cevi\'c.

\begin{Definition}[\cite{TodRho}]
 If $\alpha\leq \beta$, then
  $\varrho_1(\alpha,\beta)$ is defined recursively by\label{stevo}

  \centerline{$\varrho_1(\alpha,\alpha)=0,$}

  \centerline{$\varrho_1(\alpha,\beta)=\max\left(
     |C_\beta\cap \alpha|,\varrho_1(\alpha,\min(C_\beta\setminus
     \alpha))\right)$.} 
\end{Definition}

Following \cite{MooreL}, we let $e_\beta:\beta\to\omega$ be the
function defined by $e_\beta(\xi) = \varrho_1(\xi,\beta)$. 
Again we loosely note that in order to ensure that for some 
integer $k$ and $\omega < x< y\in \omega_1$,
$\rho_1(k,x)< \rho_1(k,y)$, we will somehow
need to have different
values for
$|C_\beta\cap k|$ for appropriate $\beta$ arising in Definition \ref{stevo}.
And now we have the definition of $\mathop{Osc}$ and can
see why we will need to  have $\rho_1(k,x)<\rho_1(k,y)$ at
critical integers $k$.

\begin{Definition} For $\alpha<\beta\in \omega_1$,
   $\mathop{Osc}(\alpha,\beta)$ is 
the set
of $\xi\in L(\alpha,\beta)\setminus
  \min(L(\alpha,\beta))$ such that
  $  
  e_\alpha(\xi^-)\leq e_\beta(\xi^-)$
  and
$  e_\alpha(\xi)>  e_\beta(\xi)$,
  where $\xi^-$ is the maximum of $L(\alpha,\beta)\cap \xi$.   

  The  value of $\mathop{osc}(\alpha,\beta)$ is equal to
  the cardinality of $\mathop{Osc}(\alpha,\beta)$.
\end{Definition}

We are ready to make our choices of $X$ and $Y$:
$$X =\{ \delta + \omega\cdot 8 :  \delta\in \Lambda''\} \ \ \mbox{and}\ \ \
 Y= \{ \delta + \omega\cdot 9 :  \delta\in \Lambda'\setminus\Lambda''\}~~.$$
For each $\min(\Lambda')\leq \alpha\in\omega_1$, let $\delta_\alpha$ be
the supremum of $\Lambda'\cap \alpha$ (thus $\delta_\alpha\in
\Lambda'$), and
let $T=\{ \alpha : \min(\Lambda')<\alpha \ \mbox{and}\
\delta_\alpha < \alpha \leq \delta_\alpha+\omega\cdot 9\}$.
 \medskip

We now give the definition of our  $ C$-sequence. 
We hope that the discussion above about our goal in ensuring that
 $L(x,y)\cap\omega$ and $L(y,x)\cap \omega$ are substantial,
 as well as occasionally needing large values of $|C_\beta\cap k|$
 for integers $k$, 
 partially
 serves to motivate the definition.    
 In
 particular,   we 
  make special coherent choices for $C_{\delta+\omega\cdot k}$ 
  for $\delta\in \Lambda'$ and $k\leq 9$.  
  We also need $\rho(\delta,\delta+\omega\cdot 8)$ and $\rho(\delta, \delta+\omega\cdot 9)$
  to be substantially different and we will accomplish this by ensuring
  that 
   the recursive definition of $L(\delta,\delta+\omega\cdot 8)$ proceeds
   through $\{\delta+ \omega\cdot 6,\delta+ \omega\cdot 4,\delta+ \omega\cdot 2\}$
   and that of $L(\delta,\delta+\omega\cdot 9)$ through
   $\{\delta+ \omega\cdot 7,\delta+ \omega\cdot 5,\delta+ \omega\cdot 3,\delta+\omega\}$.   
   At each stage an integer from $C_{\delta+\omega\cdot k}$ will
   be added to $L(\delta,\delta+\omega\cdot 8)$ but only if it is 
   larger than those added at earlier steps (see Definition \ref{lower}).
      We 
   use the integer 1000  in cases when we do not want small
   integers added to $L(\alpha,\beta)$.   
  We will use the following sets $F_1,F_2,\ldots , F_9$ in presenting our definition of our $C$-sequence:
  \begin{tabular}{ll}
  $F_1= [5,10]\cup[20,40]\cup\{90\}\cup[100,200]$&
  $F_2=[2,10]\cup \{30\}\cup[40,90]\cup \{200\}$\\  
  $F_3= [5,10]\cup[20,40]\cup\{90\}$&
  $F_4=[2,10]\cup \{30\}\cup[40,90]$\\
  $F_5=  [5,10]\cup[20,40]$&
  $F_6=[2,10]\cup \{30\}$\\  
  $F_7 =  [5,10]$ & $F_8=\emptyset$\\
   $F_9=\emptyset$
   \end{tabular}
   
Recall that $T=\{ \alpha : \min(\Lambda')<\alpha \ \mbox{and}\
\delta_\alpha < \alpha \leq \delta_\alpha+\omega\cdot 9\}$.

 \begin{Definition}  For each ordinal $\alpha <\omega_1$, we define $C_\alpha$ as follows:
\begin{enumerate}

\item $C_0=\{0\}$ and if $\omega<\alpha=\beta+1$, then  $C_\alpha=\{0,1000,\beta\} $,
\item if $0<\alpha<\omega$, then $C_\alpha = \{0,\alpha-1\}$,
\item if $ \alpha \in \Lambda\cap \min(\Lambda')$ or 
if $\alpha\in \Lambda'$, 
  then $C_\alpha$ is any suitable cofinal subset
  of $\alpha$ such that $C_\alpha \cap [0,1000] = \{0,1000\}$,
  \item  if $\alpha\in T$ and $\alpha=\delta_\alpha+\omega\cdot k$, let $j=\max(0,k-2)$ and \\
    $C_\alpha = \{0\}\cup F_k \cup \{\delta_\alpha+\omega\cdot j\}\cup 
     (\delta_\alpha+\omega\cdot (k{-}1),\delta_\alpha+\omega\cdot k)$.
\item if $\delta_\alpha <\alpha \in \Lambda\setminus T$, and
  if $\beta<\alpha\leq \beta+\omega$ for some
$\beta\in \Lambda$, then
  $C_\alpha = \{0,1000,\delta_\alpha\}\cup (\beta,\alpha)$.
\end{enumerate}
\end{Definition}

Let us note that  $1000\in C_\beta$ for all $\omega\leq\beta\notin T$ and prove the following:

\begin{Fact}  For\label{rho1}  
  all $ \beta\notin T$ and $0<k < \min(\beta,1000)$,
    $\varrho_1(k,\beta)=1$.
\end{Fact}

Proof of Fact \ref{rho1}: For each $0<n <\omega$, 
 $C_n=\{0,n-1\}$ and so, by induction on $n>k$,
  $\rho_1(k,n)=1$. If $\omega\leq \beta\notin T$,
   then $C_\beta\cap k = \{0\}$ and
    $1000=\min(C_\beta\setminus k)$. 
    Therefore, $\rho_1(k,\beta)=1$

\bigskip

It will be helpful to notice that $\max(C_\beta\cap\alpha)$ is
the minimum element of $L(\alpha,\beta)$. 

\begin{Fact} If\label{xy} $\delta\in \Lambda'$ and
 $\alpha=\delta+\omega\cdot 8\in X$
  then 
  $L(\delta_\alpha,\alpha)  = \{ 0,30,90,200\}.$
\end{Fact}
\medskip

Proof of Fact \ref{xy}:
For each $0\leq i\leq 4$, let $\beta_i=\delta+\omega\cdot (2i)$.
For each $0\leq i < 4$,
$\min(C_{\beta_{i+1}}\setminus \delta) = \beta_{i}$
and 
$\max(C_{\beta_{i+1}}\cap \delta) =\min(L(\delta,\beta_{i+1}))
> \max(C_{\beta_{i}}\cap\delta)$. 
This implies that $\max(C_{\beta_{i+1}}\cap
\delta)\cup
L(\delta,\beta_i)\subset L(\delta,
\beta_{i+1})$. Now $L(\delta,\beta_1) = \{200\}$,
$L(\delta,\beta_2)=\{90,200\}$,
$L(\delta,\beta_3)=\{30,90,200\}$, 
$L(\delta,\beta_4)=\{0,30,90,200\}$.

\bigskip

By a similar argument, which we skip, we also have:
\medskip

\begin{Fact} If\label{yx} $\delta\in \Lambda'$
and   $\alpha = \delta+\omega\cdot9$ 
  then $L(\delta_\alpha,\alpha)  = 
\{0,10,40,90,200\}.$
\end{Fact}

\medskip

Using Definition \ref{stevo} and following the ideas of the proofs of Fact \ref{xy}
and Fact \ref{yx} we have the following. We leave the simple checking to the reader.

\begin{Fact} Let $\delta\in \Lambda'$,
  and, for $k=0,\ldots,9$, let
  $\beta_k = \delta+\omega\cdot k$. Then,
  for\label{values} each $j<201$,\ \ 
  $e_{\beta_{8}}(j) = \varrho_1(j,\beta_{8}) = |C_{\beta_2}\cap j|
  =|F_2\cap j|$
  and
    $e_{\beta_9}(j)=\varrho_1(j,\beta_{9})
 = |F_1\cap j|$.
\end{Fact}

Again using   that for all $1000<\xi<\delta\in \Lambda'$,
 $1000$ is in $C_\delta$, we have that
$1000\leq  \min(L(\xi,\delta))$ and we record this next fact.

\begin{Fact} If\label{initial} $\omega\leq \xi <\delta\in \Lambda'$,
  then
 \begin{enumerate}
 \item $L(\delta,\delta+\omega\cdot 9) =
    L(\xi,\delta+\omega\cdot 9)\cap 1000$, and
\item  $L(\delta,\delta+\omega\cdot 8) =
    L(\xi,\delta+\omega\cdot 8)\cap 1000$.
  \end{enumerate}
\end{Fact}

Now we are ready to complete the proof of Lemma \ref{specialpairs}
by proving Fact \ref{5.8}.
Let $x\in X$ and $y,y'\in Y$ with $y<x<y'$.   
By definition, $x=\delta_x+\omega\cdot 8$, 
$y=\delta_y+\omega\cdot 9$, $y'=\delta_{y'}+\omega\cdot 9$
and $\delta_y < \delta_x < \delta_{y'}$.
We must prove that $
 \{30,200\}\subset \mathop{Osc}(y,x)$ and
 $\{10,90\}\subset \mathop{Osc}(x,y')    $.\\ 
By Facts \ref{xy} and \ref{yx} we have that
 $$L(\delta_x,x)=\{0,30,90,200\}\ \ \ \mbox{and}
 \ \ \  L(\delta_y,y)=L(\delta_{y'},y') = \{ 0, 10, 40,90,200\}~~.$$
 It then follows from Fact \ref{initial} that 
 $$L(y,x)=\{0,30,90,200\}\ \ \ \mbox{and}
 \ \ \  L(x,y')=  \{ 0, 10, 40,90,200\}~~.$$
 
By Facts \ref{values} and the fact that 
$$C_{\delta_x+\omega\cdot2}\cap 1000=\{0,30,200\}\cup[2,10]\cup[40,90],$$
 we have the following values:
$$e_x(0)=0 , e_x(10)= 9,
 e_x(30)=10,
 e_x(40)= 11,
 e_x(90)= 62  , e_x(200)=63.
 $$
Similarly $$C_{\delta_y+\omega}\cap 1000
=C_{\delta_{y'}+\omega}\cap 1000 = \{0,90\}\cup[5,10]\cup [20,40]\cup[100,200],$$
and we have the following values for $e_y$ (and $e_{y'}$):
$$e_y(0)=0 , e_y(10)= 6,
 e_y(30)=18,
 e_y(40)= 28,
 e_y(90)= 29  , e_y(200)=129~.
 $$
Now we verify that $\{30,200\}\subset \mathop{Osc}(y,x)$. First
let $\xi = 30\in L(y,x)=\{0,30,90,200\}$ and note
that $\xi^-=0$.  Since $e_y(\xi^-)=e_x(\xi^-)$ and
 $e_y(30)=18 > e_x(30)=10$, we have $30\in \mathop{Osc}(y,x)$.
 Similarly, with $\xi = 200$, we have $\xi^-=90$, 
  $e_y(90)=29<e_x(90)=62$, and $e_y(200)=129>e_x(200)=63$.
  
  Next we verify that $\{10,90\}\subset \mathop{Osc}(x,y')$. With
   $\xi =10\in L(x,y') = \{ 0,10,40,90,200\}$, we have
    $\xi^-=0$ and $e_x(0)=e_{y'}(0)$ while
     $e_x(10)=9>e_{y'}(10)=6$. Next let $\xi=90$
     and $\xi^-=40$. We again have $e_x(40) = 11 \leq
      e_{y'}(40)=28$ and $e_x(90)=62>e_{y'}(90)=29$.\medskip

   This completes the proof of Lemma \ref{specialpairs}.
\end{proof}

\egroup

\bgroup
\def\proofname{Proof of Lemma \ref{shrink}:\/}

\begin{proof} Let   $\mathcal A$ be an uncountable
set 
of pairwise disjoint two element 
subsets of $\omega_1$.  For $a\in \mathcal A$ we let $a(0)=\min(a)$
and $a(1)=\max(a)$. 
For each limit ordinal $\delta$, choose
any $a_\delta\in \mathcal A$ 
so that $\delta\leq \min(a_\delta) $.
Let $c_\delta=|\mathop{Osc}(a_\delta(0),a_\delta(1))\cap\delta|$.

We will
need
this next Fact 3 from \cite{MooreL} (proven in \cite{TodRho}).

\begin{Fact} For each\label{coherent} $\alpha\leq \beta\in\omega_1$,
   the set of $\xi<\alpha$ such that $e_\alpha(\xi)\neq e_\beta(\xi)$
   is finite. Also, for each $\alpha\in \omega_1$,
   the set $\{ e_\beta\restriction\alpha :
   \alpha\leq\beta\in\omega_1\}$ is countable.
\end{Fact}

Let $\{ t_\xi : \xi\in
\omega_1\}$ be an enumeration of the entire family
$\{ e_\beta\restriction \alpha : \alpha\leq \beta <\omega_1\}$.
By standard arguments,  we may choose
 a cub $C\subset\omega_1$ such that for all $\delta\in C$,
$$\{ t_\xi : \xi<\delta\} = \{ e_\beta \restriction \alpha :
 \alpha <\delta, \ \mbox{and}\ \alpha\leq \beta <\omega_1\}~~.$$
For each $\delta\in C$, there is a $\gamma_\delta\in C_\delta$ such
that
\begin{enumerate}
\item $L(\delta,a_\delta(0))\cup L(\delta,a_\delta(1))\subset \gamma_\delta$,
\item each of $e_{a_\delta(0)}\restriction[\gamma_\delta,\delta)$ and
  $e_{a_\delta(1)}\restriction [\gamma_\delta,\delta)$ is equal to
     $e_\delta\restriction[\gamma_\delta,\delta)$.
\end{enumerate}

For all
$\gamma_\delta<\xi<\delta\in C$,
 $\max(C_\delta\cap\xi)$ is the minimum element
of $L(\xi,\delta)$,  and so we have that 
$\gamma_\delta\leq \min(L(\xi,\delta))$.
By Fact \ref{helpful}, for $\gamma_\delta<\xi<\delta\in C$,
and $\delta\leq \beta$, if $\max(L(\delta,\beta))<\gamma_\delta$,
then
$L(\xi,\beta)= L(\xi,\delta)\cup L(\delta,\beta)$.

By the pressing down lemma, there is a $\gamma\in\omega_1$,
functions   $s_0,s_1 \in \omega^\gamma$,  a pair of finite sets
 $L_0,L_1\subset \gamma$, an integer $\bar c$,
and 
a stationary set $S\subset C$ satisfying that for all $\delta\in S$,
 \begin{enumerate}
 \item $\gamma_\delta=\gamma$, $L(\delta,a_\delta(0))=L_0$,
   $L(\delta,a_\delta(1))=L_1$,
 \item $e_{a_\delta(0)}\restriction\gamma = s_0$,
   $e_{a_\delta(1)}\restriction\gamma = s_1$,
 \item $e_{a_\delta(0)}\restriction[\gamma,\delta)
   =e_{a_\delta(0)}\restriction[\gamma,\delta)
     =e_{\delta}\restriction[\gamma,\delta)$,
       \item $c_\delta=\bar c$.
 \end{enumerate}

 Now let $\eta<\delta$ both be elements of $S$ and we calculate
 each of $\mathop{osc}(a_\eta(0),a_\delta(0))$
 and $\mathop{osc}(a_\eta(0),a_\delta(1))$. Let
 $\zeta=\min(L(a_\eta(0),\delta))$ and note that
  $\gamma=\gamma_\delta\leq \zeta$.
To determine the cardinality of
$\mathop{Osc}(a_\eta(0),a_\delta(0))$,
we consider each of\\
\centerline{$\mathop{Osc}(a_\eta(0),a_\delta(0))\cap
\gamma
$
\quad and \quad
$\mathop{Osc}(a_\eta(0),a_\delta(0))\setminus \gamma$.}
The set
$\mathop{Osc}(a_\eta(0),a_\delta(0))\cap
  L(\delta,a_\delta(0))
$ is empty because
 $L(a_\eta(0),\delta)\subset \gamma_\delta< a_\eta(0)$
and $e_{a_\eta(0)}\restriction\gamma
= e_{a_\delta(0)}\restriction\gamma$.  
Also, $\zeta^-<\gamma$ and so $
e_{a_\eta(0)}(\zeta^-)
= e_{a_\delta(0)}(\zeta^-)$. Let
$k =1$ if $e_{a_\eta(0)}(\zeta) >  e_{\delta}(\zeta)$ and otherwise let 
$k=0$. Since $L(a_\eta(0),a_\delta(0))\setminus \gamma$
equals 
$L(a_\eta(0),\delta)$ and 
$e_{a_\delta}(0) \restriction
     [\gamma,\delta) =
       e_{\delta}\restriction [\gamma,\delta)$,
         we have that
         $$   \mathop{Osc}(a_\eta(0),\delta)
         \subset
         \mathop{Osc}(a_\eta(0),a_{\delta}(0))\ \setminus \gamma
         \subset
         \mathop{Osc}(a_\eta(0),\delta)\cup
         \{\zeta\}.
         $$
         Similarly, we evaluate each of
         $\mathop{Osc}(a_\eta(0),a_\delta(1))\cap \gamma$
         and
                  $\mathop{Osc}(a_\eta(0),a_\delta(1))\setminus
         \gamma$. 
         We first consider $\mathop{Osc}(a_\eta(0),a_\delta(1))\setminus
         \gamma$. Since $e_{a_\delta(1)}\restriction [\gamma,\delta)$
           is also equal to $e_\delta\restriction [\gamma,\delta)$,
             and $L(a_\eta(0),a_\delta(1))\setminus \gamma$ is also
             equal to $L(a_\eta(0),\delta)$, we  have
             that
         $$   \mathop{Osc}(a_\eta(0),\delta)
         \subset
         \mathop{Osc}(a_\eta(0),a_{\delta}(1))\ \setminus \gamma
         \subset
         \mathop{Osc}(a_\eta(0),\delta)\cup
         \{\zeta\}.
         $$
         So the only possible difference between
          $$ \mathop{Osc}(a_\eta(0),a_{\delta}(1))\ \setminus \gamma 
\ \ \ \mbox{         and}\ \ \ 
                  \mathop{Osc}(a_\eta(0),a_{\delta}(0))\ \setminus
         \gamma$$
         is the singleton $\zeta$.
If $k=0$, then $\zeta\notin  \mathop{Osc}(a_\eta(0),a_\delta(1))\setminus
\gamma$.
However, if $k=1$, then $\zeta\in \mathop{Osc}(a_\eta(0),a_\delta(0))
\setminus \gamma$ but 
 $\zeta$ may or may not be
in $ \mathop{Osc}(a_\eta(0),a_\delta(1))\setminus
\gamma$.

Next, we note that $\mathop{Osc}(a_\eta(0),a_\delta(1))\cap \gamma$
equals $\mathop{Osc}(a_\eta(0),a_\eta(1))\cap \gamma$, and that
this latter set is the same for all $\eta\in S$.
This proves that 
\[
  \mathop{Osc}(a_\eta(0),a_\delta(0))\setminus \{\zeta\}
  \subset 
  \mathop{Osc}(a_\eta(0),a_\delta(1)) 
  \subset   \mathop{Osc}(a_\eta(0),a_\delta(0))~.
\]
Let the value of $c$ stated in  Lemma \ref{shrink}
 be the cardinality of
 $\mathop{Osc}(a_\eta(0),a_\eta(1))\cap \gamma$,
and this completes the proof.
\end{proof}
\egroup

\section{Remarks by the second author and by both authors of this article}
While the second author was preparing reviews of
\cite{TW} and \cite{XS}, he sent to several members of
the Carolina Seminar some related product questions and
observations, including a counterexample to the method of proof
of Theorem 3.12 in \cite{XS}.  Within a few days of each such communication,
the first author responded with proofs or indications of proofs answering
those questions and disproving Theorem 3.12. At the Seminar's meeting
the following month (on February 22, 2020), the first author presented 
 a portion of his answers, and before doing so, he informed the second
author that he would like for the two of them to write this article.
 This was agreed upon, and then some weeks later
 the first author succeeded in proving (and writing up) the results
 in Section 5.

We acknowledge with appreciation the careful, thorough and detailed review of this article and the numerous corrections and important suggestions provided by the referee.  These have led to more than 29 significant improvements in the original version of this article.

\bibliographystyle{plain}

 \end{document}